\documentclass{amsart}

\usepackage[utf8]{inputenc}
\usepackage{lmodern}
\usepackage[T1]{fontenc}
\usepackage[leqno]{amsmath}
\allowdisplaybreaks
\usepackage{amssymb} 
\usepackage{amsthm}
\usepackage{mathrsfs}
\usepackage{mathtools}
\usepackage{algorithm}
\usepackage{algpseudocode}
\usepackage{algorithmicx}
\usepackage[usenames,dvipsnames,svgnames,table]{xcolor}
\usepackage{xspace}
\usepackage{embrac} 
\ChangeEmph{(}[-.01em,.04em]{)}[.04em,-.05em]
\usepackage{enumitem}
\usepackage{graphicx}
\usepackage{float}
\usepackage{booktabs}
\usepackage{multirow}
\usepackage{tikz}
\usepackage{subcaption}
\tikzset{node distance=2cm, auto}
\usetikzlibrary{angles,quotes,calc,positioning,shapes,matrix,arrows}
\usepackage{tikz-cd}
\usetikzlibrary{}
\usepackage{tkz-euclide}
\usepackage{appendix}
\usepackage{etoolbox}
\usepackage{url}
\usepackage[hyperindex=true, colorlinks=true, linkcolor=Maroon, citecolor=Green, urlcolor=NavyBlue, breaklinks=true,linktocpage=true,linktoc=all]{hyperref}
\usepackage[english,capitalise,noabbrev]{cleveref}
\crefdefaultlabelformat{#2\textup{#1}#3}
\crefname{equation}{equation}{equations}
\usepackage[logical-quotes,backrefs,initials]{amsrefs}

\newcommand{\N}{\mathbb{N}}

\renewcommand{\P}{\mathbb{P}}
\newcommand{\RP}{\mathbb{RP}}
\newcommand{\R}{\mathbb{R}}
\renewcommand{\S}{\mathbb{S}}
\newcommand{\E}{\mathscr{E}}

\newcommand{\contained}{\subset}

\newcommand{\suchthat}{\, : \,}
\newcommand{\st}{\suchthat}

\DeclarePairedDelimiterX\norm[1]\lVert\rVert{\ifblank{#1}{\:\cdot\:}{#1}}
\DeclarePairedDelimiterX\abs[1]\lvert\rvert{\ifblank{#1}{\:\cdot\:}{#1}}
\DeclarePairedDelimiterX\set[1]{\{}{\}}{\ifblank{#1}{\: \:}{#1}}
\DeclarePairedDelimiterX\innerprod[1]\langle\rangle{\ifblank{#1}{\,\cdot,\cdot\,}{#1}}
\DeclarePairedDelimiterX\floor[1]\lfloor\rfloor{\ifblank{#1}{\:\cdot\:}{#1}}
\newcommand{\expectation}[2]{\operatorname{E}_{#1}\left[#2\right]}

\newcommand{\quotemarks}[1]{``#1''}
\newcommand{\tendsto}{\rightarrow}

\newcommand{\map}[3]{#1\colon #2 \rightarrow #3}

\newcommand{\diamondgen}{\diamond_{\mathrm{gen}}}
\newcommand{\diamondproj}{\diamond_{\P}}

\newcommand{\qde}{\textup{\textsc{qde}}\xspace}
\newcommand{\gde}{\textup{\textsc{gde}}\xspace}
\newcommand{\qgde}{\textup{\textsc{qgde}}\xspace}
\newcommand{\pgde}{\textup{\textsc{pgde}}\xspace}
\newcommand{\qpgde}{\textup{\textsc{qpgde}}\xspace}

\theoremstyle{plain}
\newtheorem{theorem}{Theorem}[section]

\newtheorem{conjecture}[theorem]{Conjecture}
\newtheorem{corollary}[theorem]{Corollary}

\newtheorem{lemma}[theorem]{Lemma}
\newtheorem{proposition}[theorem]{Proposition}

\theoremstyle{definition}

\newtheorem{definition}[theorem]{Definition}

\theoremstyle{remark}

\newcommand{\Z}{\mathbb Z}

\floatname{algorithm}{\normalfont\scshape Algorithm}
\makeatletter
\renewcommand\fnum@algorithm{\fname@algorithm~\thealgorithm:}
\makeatother

\begin{document}

	\title[Low energy points on the sphere and the real projective plane]{Low energy points on the sphere\\ and the real projective plane}
	
	\author{Carlos Beltrán}
	\address{Carlos Beltrán: Departamento de Matemáticas, Estadística y Computación, Universidad de Cantabria,  Avda. Los Castros, s/n, 39005 Santander, Spain}
	\email{beltranc@unican.es}
	
	\author{Ujué Etayo}
	\address{Ujué Etayo: Departamento de Matemáticas, Estadística y Computación, Universidad de Cantabria,  Avda. Los Castros, s/n, 39005 Santander, Spain}
	\email{etayomu@unican.es}
	
	\author{Pedro R. López-Gómez}
	\address{Pedro R. López-Gómez: Departamento de Matemáticas, Estadística y Computación, Universidad de Cantabria,  Avda. Los Castros, s/n, 39005 Santander, Spain}
	\email{lopezpr@unican.es}
	
	\date{\today{}}
	
	\thanks{The authors have been supported by grant PID2020-113887GB-I00 funded by MCIN/AEI (10.13039/501100011033). The third author has also been supported by a predoctoral grant \quotemarks{Concepción Arenal} funded by Univ. Cantabria.}
	
	

	\begin{abstract}
		We present a generalization of a family of points on $\S^2$, the Diamond ensemble, containing collections of $N$ points on $\S^2$ with very small logarithmic energy for all $N\in\mathbb{N}$. We extend this construction to the real projective plane $\RP^2$ and we obtain upper and lower bounds with explicit constants for the Green and logarithmic energy on this last space.
	\end{abstract}
	
	\maketitle
	
	\hypersetup{linkcolor=black}
	\tableofcontents
	\hypersetup{linkcolor=Maroon}


		\section{Introduction and main results}\label{sec:intro}

In this article we give an easily constructible set of reasonably well-distributed points on the sphere $\S^2$ and the real projective plane $\RP^2$.
There are many ways to define well-distributed points; see \cite{BrauchartGrabner2015} for a survey on the subject and \cite{BorodachovHardinSaff2019} for a recent monograph.
Here we work with two specific definitions: points with low logarithmic energy and low spherical cap discrepancy (the latter only for spherical points). In our case, both quantities can be computed analytically, a remarkably infrequent property among constructible sequences.
	
\subsection{Logarithmic energy}
The logarithmic energy of a set of spherical points $\omega_{N} = \{ x_{1}, \ldots  ,x_{N} \}$ is
\[
\E_{\log}(\omega_{N}) = \sum_{i \neq j} \log\| x_{i} - x_{j} \|^{-1}. 
\]
This quantity is related to the transfinite diameter and the capacity of the set by classical potential theory (see for example \cite{Landkof1972}) and can also be seen as a limit case of Riesz $s$-potentials (see for example \cite{BorodachovHardinSaff2019}). Shub and Smale found a relation between the condition number (a quantity that measures the sensitivity of zero finding) of polynomials and the logarithmic energy of the associated spherical points (see \cite{ShubSmale1993}). Inspired by this relation, they proposed a problem that is nowadays known as Smale's 7th problem \cite{Smale1998}: on input $N$, produce a set $\omega_N$ of $N$ points on the unit sphere such that
\[
\E_{\log}(\omega_{N})-m_N\leq c\log N,
\]
where $c$ is some universal constant and $m_N$ is the minimum possible value of $\E_{\log}$ among all the collections of $N$ spherical points.

	\subsection{The value of $m_N$}
	Motivated by Smale's 7th problem, the asymptotic expansion of $m_N$ has been deeply studied. The current state of the art is due to \cites{Wagner1989,RakhmanovSaffZhou1994,Dubickas1996,Lauritsen2021,Brauchart2008,BeterminSandier2018} and can be summarized in the following formula:
	\begin{equation}\label{eq:mN}
		m_N=W_{\log}(\S^2)N^2-\frac{1}{2}N\log N +C_{\log}N+o(N),
	\end{equation}
	where
	\begin{equation}\label{eq:WlogS2}
		W_{\log}(\S^2)=\frac{1}{(4\pi)^2}\int_{x,y\in\S^2}\log\norm{x-y}^{-1}\,dxdy=\frac{1}{2}-\log2
	\end{equation}
	is the continuous energy and $C_{\log}$ is a constant. Here, as usual, $o(N)$ is a function of $N$ with the property that $\lim_{N\tendsto\infty}o(N)/N=0$. From \cite{BeterminSandier2018} and \cite{Lauritsen2021} we know that
	\begin{equation}\label{eq:conjecture}
		-0.0568\dotso=-\frac{3}{4}+\log{2}\leq C_{\log}\leq 2\log{2}+\frac{1}{2}\log{\frac{2}{3}}+3\log{\frac{\sqrt{\pi}}{\varGamma(1/3)}}=-0.0556\dotso
	\end{equation}
	The upper bound has been conjectured to be an equality; see \cites{BrauchartHardinSaff2012,BeterminSandier2018} and the monograph \cite{BorodachovHardinSaff2019} for context.

\subsection{Discrepancy}
	
Let $\mu$ be the Lebesgue measure on the sphere $\S^2$. We say that a family of points $\omega_{N}$ (where each $\omega_N$ is a collection of $N$ spherical points) is asymptotically uniformly distributed if
\begin{equation*}
\lim\limits_{N\rightarrow\infty}
\frac{\#\left(\omega_{N}\cap B\right)}{N}
=
\frac{\mu(B)}{\mu\left( \mathbb{S}^{2} \right)}
\end{equation*}
for all $\mu$-continuous Borel sets $B\subset \mathbb{S}^{2}$.
The discrepancy of the family of points $\omega_{N}$ is the rate at which
\begin{equation*}
\left|
\frac{\#\left(\omega_{N}\cap B\right)}{N}
-
\frac{\mu(B)}{\mu\left( \mathbb{S}^{2} \right)}
\right|
\end{equation*}
tends to $0$ in some norm and for some suitable collection of test sets.

A spherical cap $C = C(z,t)$ centered at $z\in \mathbb{S}^{2}$ with height $t\in[-1,1]$ is  the set
\begin{equation*}
C(z,t) = \{ y\in\mathbb{S}^{2} : \left\langle z, y \right\rangle > t \},
\end{equation*}
where we denote by $\left\langle z, y \right\rangle$ the usual inner product in $\mathbb{R}^{3}$.
Let us consider the set of all the spherical caps on $\mathbb{S}^2$, which we denote by $\text{cap}$. 
Then, the spherical cap discrepancy of a set of points $\omega_{N}$ with respect to the supremum norm is defined as
\begin{equation*}
 D_{\text{sup},\text{cap}} (\omega_{N}) 
=
\sup\limits_{C \in \text{cap}}
\left|
\frac{\#\left(\omega_{N}\cap C\right)}{N}
-
\frac{\mu(C)}{\mu\left( \mathbb{S}^{2} \right)}
\right|.
\end{equation*}

\subsection{Minimal spherical cap discrepancy}
In \cites{beck_1984, Beck1984} it is proved that there exist  constants $c,c'>0$, independent of $N$, such that 
\begin{equation*}
cN^{-3/4} \leq \min_{\omega_N \subset \mathbb{S}^2} D_{\text{sup},\text{cap}}\left( \omega_{N}\right) \leq c'N^{-3/4}\sqrt{\log{N}}.
\end{equation*}
The proof of the upper bound is non-constructive. In fact, the minimal discrepancy that has been obtained for a deterministic set of points is of the order of $N^{-1/2}$ and is satisfied for the families of Fibonacci spiral points \cite{Aistleitner2012}, for any deterministic realization of the Diamond ensemble \cite{Etayo22}, for the HEALPix Points \cite{Healpix} and for families of points obtained from perturbed lattices under the Lambert projection \cite{Damir_01}.

\subsection{Distributing points on  the real projective plane}
Recall that the real projective plane $\RP^2$ can be seen as a quotient of $\S^2$ by identifying antipodal points.
One can then represent an element of the real projective plane by a unit-length vector $x\in\S^2$. 
By a slight abuse of notation, we will often treat the element itself as such a vector. 
Under this identification, the \emph{chordal distance} or \emph{projective distance} between two points $x,y\in\RP^2$ is
	\begin{equation}\label{eq:proj_distance}
		d_{\P}(x,y)=\sqrt{1-\innerprod{x,y}^2}.
	\end{equation}
Equipped with this distance, $\RP^2$ is a compact metric space (see for example \cite{ConwayHardinSloane1996}). It is easy to check that formula \eqref{eq:proj_distance} is the sine of the acute angle between the lines in $\R^3$ corresponding to $x$ and $y$.

	Given a collection $\omega_N=\set{x_1,\dotsc,x_N}\contained\RP^2$ of points in the real projective plane, we define its projective logarithmic energy as
	\begin{equation}\label{eq:log_energy_proj}
		\E_{\log}^{\P}(\omega_N)=-\sum_{i\neq j}\log d_{\P}(x_i,x_j)=-\sum_{i\neq j}\log\sqrt{1-\innerprod{x_i,x_j}^2}.
	\end{equation}
	
\subsection{Logarithmic and Green energies}	
	
	Both in the sphere and the real projective plane the logarithmic energy coincides, up to multiplicative and additive constants, with the Green energy of these spaces seen as Riemannian manifolds, as studied in \cites{BeltranCorralCriado2019,Andersonetal}. In particular, in the real projective plane we have
	\begin{equation}\label{eq:Greenlog}
		\E_{G}^{\P}(\omega_N)=\frac{1}{2\pi}\left(\E_{\log}^{\P}(\omega_N)- N(N-1)W_{\log}(\RP^2)\right),
	\end{equation}
	where 
	\begin{equation}\label{eq:WlogRP2}
		W_{\log}(\RP^2)=1-\log 2
	\end{equation}
	 is the continuous energy of the real projective plane (see \cite{Andersonetal}). Note that in \cite{Andersonetal} the authors use a different convention: they take the Riemannian structure normalized so that the total volume of $\RP^2$ is equal to $1$, which results in a lacking constant $1/(2\pi)$ that appears in our formulas but not in theirs. Taking this into account, it has been proved in \cite{Andersonetal} that
	\begin{equation}\label{eq:minGreenEnergy}
		\min\limits_{\omega_N\contained\RP^2}\E_{G}^{\P}(\omega_N)=-\frac{1}{4\pi}N\log{N}+O(N),
	\end{equation} 
	which readily yields
	\begin{equation}\label{eq:cotaeprojanderson}
		\min_{\omega_N\contained\RP^2}\E_{\log}^{\P}(\omega_N)=W_{\log}(\RP^2)N^2-\frac{1}{2}N\log{N}+O(N).
	\end{equation}

\subsection{Main results and a conjecture}\label{subsec:mainresults}

In \cites{BeltranEtayo2020,9031002} the authors define a family of points called the Diamond ensemble, which,	
briefly speaking, is a sequence of sets of points on $\S^2$, depending on many parameters, that are placed on the sphere as follows. First choose $p$ parallels. Then, for every parallel $j$, choose $r_j$ points and allocate them in the vertices of a regular $n$-gon on that parallel rotated by a random phase $\theta_j\in[0,2\pi]$. Finally, add the north and south poles.

The main advantage of the sets of points coming from the Diamond ensemble is that both the logarithmic energy and the spherical cap discrepancy can be computed (see \cites{BeltranEtayo2020,9031002,Etayo22}). Its main drawback is that it is not defined for any choice of $N$, but only for some infinite subsequences of natural numbers.

In this paper, we present a generalization of the Diamond ensemble that can be defined for any number of points and that, at the same time, maintains the property that the logarithmic energy and the spherical cap discrepancy can be analytically computed and they are actually the same as in the original Diamond ensemble. This is established in the following results.
	
\begin{theorem}\label{thm:finalmodel2}
	For any $N\in\mathbb{N}$ there exists an explicit set $\diamondgen(N)$ of $N$ points on $\mathbb{S}^2$, which depends on certain random parameters, whose expected logarithmic energy is
	\begin{align*}
		W_{\log}(\S^2)N^2-\frac{1}{2}N\log N+c_{\diamond}N+O\left(\sqrt{N}\log N\right),
	\end{align*}
	where $c_{\diamond}=-0.049222\dotso$ is as in \cref{eq:Cdiamond} below. Moreover, for any choice of the random parameters, its spherical cap discrepancy satisfies
	\begin{equation*}
		\frac{c_{1}}{\sqrt{N}}
		\leq
		D_{\text{\rm sup},\text{\rm cap}}\left(\diamondgen(N)\right)
		\leq
		\frac{c_{2}}{\sqrt{N}},
	\end{equation*}
	for some universal constants $c_1,c_2>0$.
\end{theorem}

The explicit construction of the set $\diamondgen(N)$ is given in Section \ref{sec:diamond}.

Now we move to the real projective plane. Through its relationship with the sphere, we will easily establish a lower bound for the projective logarithmic energy.
	\begin{proposition}\label{prop:lowerboundproj}
		The projective logarithmic energy of any collection of $N$ points in $\RP^2$ satisfies
		\begin{align}
			\min_{\omega_N\contained\RP^2}\E^\P_{\mathrm{log}}(\omega_N)&\geq W_{\log}(\RP^2)N^2-\frac{1}{2}N\log{N}+\left(C_{\log}-\frac{1}{2}\log{2}\right)N+o(N), \label{eq:lowerboundProjectiveEnergy}
		\end{align}
with $W_{\log}(\RP^2)$ given by \eqref{eq:WlogRP2}. 
Note that $C_{\log}-\frac{1}{2}\log{2}\geq-\frac{3}{4}+\frac{1}{2}\log{2}=-0.403426\ldots$
\end{proposition}
\Cref{prop:lowerboundproj} above improves the lower bound in \cref{eq:minGreenEnergy}.
For the upper bound, we define a set of points on $\RP^2$ by identifying antipodal points on $\S^2$ for which we are able to compute its logarithmic energy.
\begin{theorem}\label{thm:maintheorem2proj}
For any $N\in\N$ there exists an explicit set $\diamondproj(N)$ of $N$ points on $\RP^2$, which depends on certain random parameters, whose expected logarithmic energy is
$$
W_{\log}(\RP^2)N^2-\frac12N\log N+\left(c_\diamond -\frac12\log2\right)N+O\left(\sqrt{N}\log N\right).
$$
Note that $c_\diamond -\frac12\log2=-0.395795\ldots$
	\end{theorem}

For completeness, we write down the corresponding values of the Green energy both for the upper bound and for the lower bound.

\begin{corollary}\label{cor:boundsgreen}
	The minimum value of the Green energy of a set $\omega_N$ of $N$ points on $\RP^2$ satisfies
	\begin{multline*}-\frac{N}{4\pi}\log{N}+\frac{1}{2\pi}\left(\frac{1}{4}-\frac{1}{2}\log{2}\right)N +o(N)\\\leq \min_{\omega_N\contained\RP^2}\E_{G}^{\P}(\omega_N)\leq\\ -\frac{N}{4\pi}\log{N}+ \frac{1}{2\pi}\left(c_{\diamond}+1-\frac{3}{2}\log{2}\right)N+o(N),
	\end{multline*}
	with $c_{\diamond}$ as in \eqref{eq:Cdiamond}. The upper bound is the expected value of the Green energy of $\diamondproj(N)$.
\end{corollary}

\begin{proof}
	Immediate from \eqref{eq:Greenlog}, \cref{prop:lowerboundproj} and \cref{thm:maintheorem2proj}.
\end{proof}

The facts that projective points correspond to antipodal spherical points and that the bounds in \cref{thm:maintheorem2proj} and \cref{prop:lowerboundproj} are so close give credit to the following open question.
	\begin{conjecture}
The minimum value of the logarithmic energy among all the collections $\omega_N$ of antipodal points in $\S^2$ \embparen{i.e., $p\in\omega_N$ implies $-p\in\omega_N$} satisfies
		\begin{equation*}
			\min_{\omega_N\contained \S^2}\E_{\log}(\omega_N)=W_{\log}(\S^2)N^2-\frac{1}{2}N\log{N}+C_{\log}N+o(N).
		\end{equation*}
	\end{conjecture}

\subsection{Structure of the paper}
	
In \cref{sec:diamond} we recall the construction called Diamond ensemble, we generalize said procedure and prove that is valid for any number of points, and we prove \cref{thm:finalmodel2}. 
In \cref{sec:projplane} we present some results concerning the logarithmic energy of projective points coming from the general construction described in \cref{sec:diamond}, we define the projective Diamond ensemble and we prove \cref{prop:lowerboundproj,thm:maintheorem2proj}. 
Finally, \cref{sec:proofs} is devoted to the proof of some auxiliary results.


	\section{The generalized Diamond ensemble}\label{sec:diamond}
	
\subsection{A general construction}
	
	 Consider the construction mentioned at the beginning of \cref{subsec:mainresults}, which was studied in \cite{BeltranEtayo2020} and is similar in spirit to Equal Area Partition points \cite{HMS}. A precise definition of this set is the following.
	 
	 \begin{definition}\label{def:Omega}
	 	Let $\varOmega(p,\set{r_j},\set{z_j})$ be the following set of points:
	 	\begin{equation*}
	 		\varOmega(p,\set{r_j},\set{z_j})=
	 		\begin{cases}
	 			\mathscr{N}=(0,0,1),\\
	 			x^{i}_{j}=\left(\sqrt{1-z_j^2}\cos\left(\frac{2\pi i}{r_j}+\theta_j\right),\sqrt{1-z_j^2}\sin\left(\frac{2\pi i}{r_j}+\theta_j\right),z_j\right),\\
	 			\mathscr{S}=(0,0,-1),
	 		\end{cases}
	 	\end{equation*}
 		where $p$ is the number of parallels, $r_j$ is the number of roots of unity that we consider in the parallel $j$, with $1\leq j\leq p$, $1\leq i\leq r_j$, and $0\leq \theta_j<2\pi$ is a random phase in the parallel $j$.
	 \end{definition} 
 	
 	It turns out that there exists a unique set of heights $z_1,\ldots,z_p$ that minimize the expected logarithmic energy of the corresponding points (see \cref{eq:heights} below). We denote that set of points by $\Omega(p,\set{r_j})$, and then we have the following result.
	
	\begin{theorem}[\cite{BeltranEtayo2020}*{Thm. 2.6}]\label{thm:2.6}
		Let $p=2M-1$ be an odd integer. If $r_{j}=r_{p+1-j}$ and the $z_j$ are chosen as in \cref{eq:heights}, then we have
		\begin{align*}
			\MoveEqLeft \expectation{\theta_1,\dotsc,\theta_{2M-1}\in[0,2\pi]}{\E_{\log}(\Omega(2M-1,\set{r_j}))}\\
			&= -(N-1)\log{4}-r_M\log{r_M}-2\sum_{j=1}^{M-1}r_j\log{r_j}\\ &\phantom{=}-(N-1)\sum_{j=1}^{M-1}r_j(1-z_j)\log(1-z_j)-(N-1)\sum_{j=1}^{M-1}r_j(1+z_j)\log(1+z_j).
		\end{align*}
		Here, $N=2+r_1+\dotsb+r_p$ is the total number of points.
	\end{theorem}

The Diamond ensemble is obtained from this general idea, imposing that the number of points in each parallel must satisfy a continuous piecewise linear equation with integer coefficients together with some extra hypotheses. With a sensible choice of the continuous piecewise linear function (the \emph{quasioptimal Diamond ensemble} (\qde) of \cite{BeltranEtayo2020}), the expected logarithmic energy is
	\begin{equation*}
		W_{\log}(\S^2)-\frac{1}{2}N\log N+c_{\diamond}N+o(N),
	\end{equation*}
	where $c_{\diamond}=-0.049222\dotso$ satisfies
	\begin{align}\label{eq:Cdiamond}
		14340\,c_{\diamond}&=19120\log239 - 2270\log227 - 1460\log73 - 265\log53- 1935\log43\nonumber\\
					 & - 930\log31 - 1710\log19 - 1938\log17 + 19825\log13+ 1750\log7 \\
					 & - 4250\log5 - 131307\log3 + 56586\log2 - 7170\nonumber.
	\end{align}
	This constant is approximately $0.0058$ far from the value conjectured in \cites{BeterminSandier2018,BrauchartHardinSaff2012}.

As said before, the main disadvantage of the Diamond ensemble is that it can only be constructed for certain values of $N$ (those of the form $N=239m^2+2$, with $m$ a positive integer, in the case of the \qde). The reason for this restriction is the need of the piecewise linear function to be continuous.

\subsection{The generalized Diamond ensemble} We now recall the definition of the Diamond ensemble, introducing the unique but fundamental difference that the piecewise linear function may have discontinuities as long as some quite mild hypotheses are satisfied, and we prove that the expected logarithmic energy can still be computed. This will thus allow us to extend the construction to any number of points.

\begin{definition}[A constant associated to a piecewise linear mapping]\label{def:Crx}
	Let $M$ be a positive integer and let $\map{r}{[0,M]}{\R}$ be the piecewise linear function
	\begin{equation*}
		r(x)=
		\begin{cases}
			\alpha_1+\beta_1 x & \text{if $0=t_0< x\leq t_1$},\\
			\alpha_2+\beta_2 x & \text{if $t_1< x\leq t_2$},\\
			\vdots & \vdots\\
			\alpha_n+\beta_n x & \text{if $t_{n-1}< x\leq t_n=M-1$},\\
			\alpha_{n+1}+\beta_{n+1} x& \text{if $M-1< x\leq M$},
		\end{cases}
	\end{equation*}
	where $n,t_{\ell},\alpha_\ell,\beta_\ell\in\Z$ are nonnegative and $\beta_1>0$. We define the \emph{constant associated to $r(x)$} by
	\begin{equation*}
		C_{r(x)}=\max_{1 \leq \ell \leq n+1}(\alpha_\ell/M,\beta_\ell,M/t_1).
	\end{equation*}
\end{definition}

\begin{definition}[Generalized Diamond ensemble]\label{def:generalizedDiamond}
	For any choice of $r(x)$ as in \cref{def:Crx}, the \emph{generalized Diamond ensemble} (\gde) is a collection of
	\begin{equation*}
		N=2+(\alpha_{n+1}+\beta_{n+1} M)+2\sum_{\ell=1}^{n}\sum_{j=t_{\ell-1}+1}^{t_\ell}(\alpha_\ell+\beta_\ell j)
	\end{equation*}
	points consisting of
	\begin{enumerate}
	  \item The north and south poles,
	  \item $r_j=r(j)$ equally spaced points, rotated by a random phase $\theta_j$, on the parallel of height
	      \begin{equation}\label{eq:heights}
	      z_j=1-\frac{1+r(j)+2\displaystyle\sum_{k=1}^{j-1}r(k)}{N-1},
	      \end{equation}
	      for $1\leq j\leq M-1$,
	  \item $r_M=r(M)$ equally spaced points, rotated by a random phase $\theta_M$, on the equator,
	  \item and $r_j$ equally spaced points, rotated by a random phase $\theta_{2M-j}$, on the parallel of height $-z_j$, for $1\leq j\leq M-1$.
	\end{enumerate}
	As mentioned above, the choice \eqref{eq:heights} is due to \cite[Prop. 2.5]{BeltranEtayo2020}, where it is proved that for any choice of $r_1,\ldots,r_M$ there exists a unique set of heights (given by \eqref{eq:heights}) that minimizes the expected value of the logarithmic energy of the corresponding spherical points. We denote this set of points by $\diamondgen(N)$, omitting in the notation the dependence on all the parameters $M,n,t_\ell,\alpha_\ell,\beta_\ell$.
\end{definition}
Note that
		\begin{equation}\label{eq:boundsN}
\frac{M^2}{2C_{r(x)}^2}\leq \frac{t_1(t_1+1)}{2}\leq  \sum_{j=1}^{t_1}\beta_1 j\leq N\leq 2+2\sum_{j=1}^{M}(C_{r(x)}M+C_{r(x)}j)\leq 5C_{r(x)}M^2.
		\end{equation}
		We denote by $N_\ell$ the total number of points up to parallel $t_{\ell-1}$ (included) without the north pole, that is, $N_\ell=\sum_{j=1}^{t_{\ell-1}}r_j$. Note that, if $j\in(t_{\ell-1},t_{\ell}]$, then
		\begin{equation}\label{eq:z_j}
			z_j=1-\frac{1+r_j+2\sum_{k=1}^{j-1}r_k}{N-1}=1-\frac{1+2N_\ell-r_j+2 \sum_{k=t_{\ell-1}+1}^{j}(\alpha_\ell+\beta_\ell k)}{N-1}.
		\end{equation}
		Hence, defining the function $\map{z}{[0,M]}{\R}$ by $z(0)=1$ and piecewise in each interval $(t_{\ell-1},t_\ell]$ by the degree $2$ polynomial
		\begin{multline}\label{eq:z_l(x)}
				z_\ell(x)=
				1-\frac{1+2N_\ell-(\alpha_\ell+\beta_\ell x)+2\alpha_\ell(x-t_{\ell-1})+\beta_\ell(x+t_{\ell-1}+1)(x-t_{\ell-1})}{N-1},
		\end{multline}
		we note that $z_j=z(j)$ for all $1\leq j\leq M$. We extend $r(x)$ and $z(x)$ to the whole domain $[0,2M]$ by $r(2M-x)=r(x)$ and  $z(2M-x)=-z(x)$.

	\subsection{An exact formula for the expected logarithmic energy of the generalized Diamond ensemble}
	
	The expected value of the logarithmic energy of $\diamondgen(N)$ can be studied using \cref{thm:2.6}. We now write the sums in that theorem as instances of a composite trapezoidal rule. Recall that for a function $\map{f}{[a,b]}{\R}$, with $a<b$ integers, the composite trapezoidal rule is
	\begin{equation*}
		T_{[a,b]}(f)=\frac{f(a)+f(b)}{2}+\sum_{j=a+1}^{b-1}f(j).
	\end{equation*}
	We immediately deduce the following result, which is a version of \cite[Cor. 3.2]{BeltranEtayo2020} for not necessarily continuous piecewise linear choices of $r(x)$.
	\begin{corollary}\label{cor:exactEnergy}
		The expected logarithmic energy of points drawn from the {\gde} equals
		\begin{align*}
			\MoveEqLeft \expectation{\theta_{1},...,\theta_{2M-1} \in [0,2\pi]}{\E_{\log}(\diamondgen(N))}=-(N-1)\log{4}-r(M)\log{r(M)}\\
			&-2\sum_{\ell=1}^{n}\left(\frac{f_\ell(t_{\ell-1}+1)+f_\ell(t_{\ell})}{2}+T_{[t_{\ell-1}+1,t_{\ell}]}(f_{\ell})\right)\\
			&-(N-1)\sum_{\ell=1}^{n}\left(\frac{g_\ell(t_{\ell-1}+1)+g_\ell(t_{\ell})}{2}+T_{[t_{\ell-1}+1,t_{\ell}]}(g_{\ell})\right)\\
			&-(N-1)\sum_{\ell=1}^{n}\left(\frac{h_\ell(t_{\ell-1}+1)+h_\ell(t_{\ell})}{2}+T_{[t_{\ell-1}+1,t_{\ell}]}(h_{\ell})\right),
		\end{align*}
		where for $1\leq \ell\leq n$ the functions $f_\ell,g_\ell,h_\ell$ are defined in the interval $(t_{\ell-1},t_\ell]$ and satisfy
		\begin{align*}
			f_\ell(x)&=(\alpha_\ell+\beta_\ell x)\log(\alpha_\ell+\beta_\ell x),\\
			g_\ell(x)&=(\alpha_\ell+\beta_\ell x)(1-z_\ell(x))\log(1-z_\ell(x)),\\
			h_\ell(x)&=(\alpha_\ell+\beta_\ell x)\left(1+z_\ell(x)\right)\log(1+z_\ell(x)).\\
		\end{align*}
	\end{corollary}

	\subsection{An asymptotic formula for the expected logarithmic energy of the generalized Diamond ensemble}
	
	We now recall lemmas \cite{BeltranEtayo2020}*{Lemmas 3.3--3.5}, applied to $[t_{\ell-1}+1,t_\ell]$ instead of $[t_{\ell-1},t_\ell]$ as in that paper.
	
	\begin{lemma}\label{lemma:f}
		For $1\leq \ell\leq n$ we have
		\begin{equation*}
			\abs*{T_{[t_{\ell-1}+1,t_{\ell}]}(f_{\ell})-\int_{t_{\ell-1}+1}^{t_{\ell}}f_{\ell}(x)dx}\leq 
3C_{r(x)}M\log(2C_{r(x)}M).
		\end{equation*}
	\end{lemma}
	
	\begin{lemma}\label{lemma:g}
		The following inequality holds for $1\leq \ell\leq n$:
		\begin{equation*}
			\abs*{T_{[t_{\ell-1}+1,t_{\ell}]}(g_{\ell})-\int_{t_{\ell-1}+1}^{t_{\ell}}g_{\ell}(x)dx - \frac{g'_{\ell}(t_{\ell})-g'_{\ell}(t_{\ell-1}+1)}{12}}\leq \frac{\Lambda\log M}{M},
		\end{equation*}
		for some constant $\Lambda$ that depends only on $C_{r(x)}$.
	\end{lemma}
	
	\begin{lemma}\label{lemma:h}
		The following inequality holds for $1\leq \ell\leq n$:
		\begin{equation*}
			\abs*{T_{[t_{\ell-1}+1,t_{\ell}]}(h_{\ell})-\int_{t_{\ell-1}+1}^{t_{\ell}}h_{\ell}(x)dx - \frac{h'_{\ell}(t_{\ell})-h'_{\ell}(t_{\ell-1}+1)}{12}}\leq \frac{\Lambda}{M},
		\end{equation*}
		for some constant $\Lambda$ that depends only on $C_{r(x)}$.
	\end{lemma}
	
	The following theorem is thus proved.
	
	\begin{theorem}\label{thm:asymptoticEnergy}
		For the {\gde} we have
		\begin{multline*}
			\expectation{\theta_{1},...,\theta_{2M-1} \in [0,2\pi]}{\E_{\log}(\diamondgen(N))}=-(N-1)\log{4}\\
			-2\sum_{\ell=1}^{n}\left(\frac{f_\ell(t_{\ell-1}+1)+f_\ell(t_{\ell})}{2}+\int_{t_{\ell-1}+1}^{t_{\ell}}f_{\ell}(x)\,dx\right)\\
			-(N-1)\sum_{\ell=1}^{n}\left(\frac{g_\ell(t_{\ell-1}+1)+g_\ell(t_{\ell})}{2}+\frac{g_\ell'(t_\ell)-g_\ell'(t_{\ell-1}+1)}{12}+\int_{t_{\ell-1}+1}^{t_{\ell}}g_{\ell}(x)\,dx\right)\\
			-(N-1)\sum_{\ell=1}^{n}\left(\frac{h_\ell(t_{\ell-1}+1)+h_\ell(t_{\ell})}{2}+\frac{h_\ell'(t_\ell)-h_\ell'(t_{\ell-1}+1)}{12}+\int_{t_{\ell-1}+1}^{t_{\ell}}h_{\ell}(x)\,dx\right)\\
			+{\mathrm{Error}},\quad \abs{\mathrm{Error}}\leq \Lambda nM\log M,
		\end{multline*}
		where for $1\leq \ell \leq n$ the functions $f_{\ell},g_{\ell},h_{\ell}$ are as in \cref{cor:exactEnergy} and the constant $\Lambda$ depends only on $C_{r(x)}$.
	\end{theorem}


	\subsection{The quasioptimal Diamond ensemble for any number of points}\label{subsec:quasioptimalDiamond}
	
We now generalize the \qde of \cite{BeltranEtayo2020}*{Subsec. 4.3} for any choice of $N$. Let $M=7m$ with $m$ a positive integer. Let also $\gamma\in\{0,\ldots,800\}$ be an integer and $\varepsilon\in\{-1,0,1\}$ (we could be more general and say let $\gamma$ grow up to $1000$ and $|\varepsilon|\leq 10$, but we do not need it), and let $\delta\in[6+1/m,7-1/m]$ be such that $\delta m\in\Z$. Consider the \gde with
	\begin{equation}\label{eq:r(x)_Quasioptimal_gen}
		r(x)=
		\begin{cases}
			6x					& \text{if $0< x\leq 2m$},\\
			2m+5x				& \text{if $2m< x\leq 3m$},\\
			5m+4x				& \text{if $3m< x\leq 4m$},\\
			9m+3x				& \text{if $4m< x\leq 5m$},\\
			14m+2x				& \text{if $5m< x\leq 6m$},\\
			20m+x+\gamma				& \text{if $6m< x\leq \delta m$},\\
			20m+x+\gamma+1				& \text{if $\delta m< x\leq 7m-1$},\\
			20m+x+\gamma+1+\varepsilon			& \text{if $7m-1< x\leq 7m$}.
		\end{cases}
	\end{equation}
	Since $z(x)$ is defined by \eqref{eq:z_l(x)}, we can write it down explicitly with a simple, closed formula, in order to generate the points . We compute the associated constant:
	\begin{equation*}
		C_{r(x)}=\max_{1 \leq \ell \leq n+1}(20/7+(\gamma+1+\abs{\varepsilon})/(7m),5,7/2)\leq200.
	\end{equation*}
	The total number of points of the associated ensemble is
	\begin{equation*}
		N=239m^2+2+(2m-1)\gamma+2(7m-1-\delta m) +1+\varepsilon.
	\end{equation*}
	We now prove that this construction is general enough as to describe a collection of $N$ spherical points for any $N\geq241$ whose expected energy and spherical cap discrepancy are as in \cref{thm:finalmodel2}. We call this set of points the \emph{quasioptimal generalized Diamond ensemble} (\qgde) and its specific construction is given in \cref{subsec:proof1}. Note that, since we are primarily interested in the asymptotic behavior of the logarithmic energy, there is no loss of generality in considering sufficiently large $N$. For $N$ smaller than any given constant we can easily adapt the construction of \cite[Sec. 3.1]{BeltranEtayo2020} taking, for example, the integer part of the $r_j$ there defined\footnote{In the web page \url{https://personales.unican.es/beltranc/CollectionsofPoints.html}, the interested reader can consult, for a certain range of values of $N$, the energy and exact positions of the $N$ points on $\S^2$ using this approach.}, and computing the energy using \cite[Thm. 2.6]{BeltranEtayo2020}.

\subsection{Proof of \cref{thm:finalmodel2}}\label{subsec:proof1}

		We now prove that for any $N\in\N$, with $N\geq241$, one can choose $m\geq1$, $\gamma\in\set{0,1,\ldots,800}$, $\varepsilon\in\{-1,0,1\}$ and $\delta\in[6+1/m,7-1/m]$, with $\delta m\in\Z$, in equation \eqref{eq:r(x)_Quasioptimal_gen} such that the corresponding \gde has exactly $N$ points. There are many ways to achieve this goal yielding the expected energy of \cref{thm:finalmodel2}. A very simple choice is as follows.
		
		\begin{enumerate}
			\item If $N=239m^2+2$ for some $m\in\N$, with $m\geq1$, then we can take
			\begin{equation*}
				m=\sqrt{\frac{N-2}{239}},\quad \gamma=0,\quad \delta m= 7m-1,\quad \varepsilon=-1.
			\end{equation*}
			\item If $N$ is not of the form $239m^2+2$, let
			\begin{align*}
				m&=\floor*{\sqrt{\frac{N-2}{239}}},\quad \gamma=\floor*{\displaystyle\frac{N-239m^2-2}{2m-1}}\leq 800,\\ \eta&=(N-239m^2-2)\bmod (2m-1).
			\end{align*}
			Note that we are distributing the $N-239m^2-2$ remaining points by allocating $\gamma$ of them in each of the $2m-1$ central parallels of the sphere, and $\eta$ is the number of points left after doing that.
			\begin{enumerate}
				\item If $\eta=0$, then we do not have to add any extra points to any parallel, so we can take
				\begin{equation*}
					\delta m=7m-1,\quad  \varepsilon=-1.
				\end{equation*}
				\item If $\eta\neq 0$ is odd, we add an extra point to the central $\eta$ parallels of the sphere, that is, we take
				\begin{equation*}
					\delta m=7m-\frac{\eta-1}{2}-1,\quad \varepsilon=0.
				\end{equation*}
				\item If $\eta\neq 0$ is even, we add an extra point to the central $\eta-1$ parallels of the sphere, and another additional point to the equator. More precisely, we take
				\begin{equation*}
					\delta m=7m-\frac{\eta-2}{2}-1,\quad  \varepsilon=1.
				\end{equation*}
			\end{enumerate}
		\end{enumerate}
		
		\subsubsection{Proof of the expected energy in \cref{thm:finalmodel2}}
		
		All the integrals and derivatives in \cref{thm:asymptoticEnergy} can be computed, which is a huge calculus task. We have done these computations using the computer algebra system SageMath \cite{SageMath}. Moreover, since $C_{r(x)}$ is bounded above by $100$ and $n\leq 7$ independently of all the choices, the error term is of the form $\Lambda M\log M$, i.e., $\Lambda\sqrt{N}\log N$, with $\Lambda$ some universal constant. \qed

	\subsubsection{Proof of the discrepancy bounds in \cref{thm:finalmodel2}}
	
	We follow the proof of \cite[Thm. 1.1]{Etayo22}. It is enough to observe that \cite[Lemma 2.4]{Etayo22} is also satisfied by our construction by \cref{eq:boundsN}. We can then follow the proofs of \cite[Thms. 1.6 and 1.7]{Etayo22}, concluding the same result.

\section{The Diamond ensemble in the real projective plane}\label{sec:projplane}
	
	From now on we show how to use the ideas already discussed for $\S^2$ to generate collections of points on the real projective plane $\RP^2$ with small logarithmic energy. 

\subsection{Relation between the energies of spherical and projective points}

Given $N$ projective points $\omega_N=\set{x_1,\dotsc,x_N}\contained\RP^2$, one can consider the associated collection of $2N$ spherical points $\{\omega_N,-\omega_N\}$. An elementary computation using $\norm{x_i\pm x_j}^2=2\pm2\innerprod{x_i,x_j}$ yields
\begin{equation}\label{eq:energycomparison}
\E_{\log}^{\P}(\omega_N)=\frac12\E_{\log}(\{\omega_N,-\omega_N\})+N^2\log{2}.
\end{equation}

\begin{theorem}\label{thm:energyAntipodal}
	Let $p=2M-1$ be an odd integer. Let $\Omega(p,\{r_j\},\{z_j\})$ be as in \cref{def:Omega}, and let $\Omega^{a}(p,\{r_j\},\{z_j\})$ be the corresponding set of points assuming $\theta_{2M-j}=\theta_{j}+\pi$. Then,
	\begin{multline*}
		\expectation{\theta_1,\dotsc,\theta_M\in[0,2\pi]}{\E_{\log}(\Omega^{a}(2M-1,\{r_j\},\{z_j\}))}\\=\expectation{\theta_1,\dotsc,\theta_{2M-1}\in[0,2\pi]}{\E_{\log}(\Omega(2M-1,\{r_j\},\{z_j\}))}\\-2\sum_{j=1}^{M-1}r_j\log\left(1-\frac{(z_j-1)^{r_j}}{(z_j+1)^{r_j}}\right).
	\end{multline*}
	Moreover, if the $r_j$ satisfy $A^{-1}j\leq r_j\leq Aj$ for some positive constant $A$ and the $z_j$ are chosen as in \eqref{eq:optimalheights_proj}, then
	\begin{equation*}
		\abs*{-2\sum_{j=1}^{M-1}r_j\log\left(1-\frac{(z_j-1)^{r_j}}{(z_j+1)^{r_j}}\right)}\leq BM,
	\end{equation*}
	where $B$ is a constant depending only on $A$.
\end{theorem}

\begin{proof}
	In \cite[Prop. 2.4]{BeltranEtayo2020}, the authors prove that the expected logarithmic energy of the set $\Omega(p,\{r_j\},\{z_j\})$ can be computed as a sum of three quantities: the energy between every point $x_{j}^{i}$ and the poles plus the energy between the poles, the energy of the scaled roots of unity for every parallel, and the energy between the points of every pair of parallels. Clearly, the condition $\theta_{2M-j}=\theta_{j}+\pi$ affects only the latter. Since said quantity is originally computed as an expectation with respect to the random parameters $\theta_1,\dotsc,\theta_{2M-1}$, to prove \cref{thm:energyAntipodal} we just need to:
\begin{itemize}	
\item	Subtract the expected value of the energy between each pair of parallels $j$ and $2M-j$, which by \cite[Cor. 2.3]{BeltranEtayo2020} is
	\begin{multline*}
		2\expectation{\theta_j,\theta_{2M-j}\in[0,2\pi]^2}{-\sum_{i,k=1}^{r_j}\log\norm*{x^{k}_{j}-x^{i}_{2M-j}}}\\
		=-r_j^2\log(1+z_j^2+2z_j)=-2r_j^2\log(1+z_j),
	\end{multline*}
\item and add the corresponding exact energy between those parallels, which will be computed in \cref{lemma:energy_parallels} and has the following expression:
		\begin{equation*}
			-2\sum_{i,k=1}^{r_j}\log\norm{x^{k}_{j}-x^{i}_{2M-j}}=
			-2r_j^2\log(1+z_j)-2r_j\log\left(1-\frac{(z_j-1)^{r_j}}{(z_j+1)^{r_j}}\right).
		\end{equation*}
\end{itemize}
The theorem follows using \cref{lemma:1}.\qedhere
	
\end{proof}

\subsection{Proof of \cref{prop:lowerboundproj}}
The lower bound for the logarithmic energy on the sphere readily implies \cref{prop:lowerboundproj} using \eqref{eq:mN} to find a lower bound for the spherical energy in \eqref{eq:energycomparison} and plugging in the values of $W_{\log}(\S^2)$ and $W_{\log}(\RP^2)$. \qed

\subsection{The general projective construction}\label{sec:Omega_proj}

We propose a construction for projective points in the spirit of the Diamond ensemble whose energy can be computed and, again, yields values which are very close to the lower bound.

\begin{definition}\label{def:generalstructureDiamondproj}
	Given positive integers $r_1,\ldots,r_{M-1},r_M$, with $r_M$ an even number, and heights $1>z_1>\cdots>z_{M-1}>z_M=0$, the set $\Omega_\P(M,\set{r_j},\set{z_j})$ is a collection of $N=1+r_1+\ldots+r_{M-1}+r_M/2$ points consisting of
	\begin{enumerate}
	  \item The north pole,
	  \item $r_j$ equally spaced points, rotated by a random phase $\theta_j$, on the parallel of height $z_j$ for $1\leq j\leq M-1$,
	  \item and the following $r_M/2$ equally spaced points on the half-equator:
	  $$
	  \left\{\left(\cos\frac{2\pi k}{r_M},\sin\frac{2\pi k}{r_M},0\right)\st 0\leq k\leq \frac{r_M}{2}-1  \right\}.
	  $$
	\end{enumerate}
\end{definition}

	In the (spherical) Diamond ensemble, the optimal $z_j$ given in \eqref{eq:heights} were found directly by computing the derivatives with respect to $z_j$ and setting them equal to $0$. The same method does not yield a solvable system in the projective case, but we can still use a simple formula for the heights: inspired by \eqref{eq:energycomparison}, we take the $z_j$ as the optimal heights above the equator for the associated set $\Omega(p=2M-1,\set{r_1,\ldots,r_{M-1},r_M,r_{M-1},\ldots,r_1})$ (which contains $2N$ spherical points), that is:
	\begin{equation}\label{eq:optimalheights_proj}
		z_{j}=1-\frac{1+r_j+2\displaystyle\sum_{k=1}^{j-1}r_k}{2N-1}.
	\end{equation}
	With this choice of the $z_j$ we denote the corresponding collection of points by $\Omega_\P(M,\{r_j\})$ and we have the following result, which is the projective version of \cref{thm:2.6}.
	
	\begin{theorem}\label{thm:projenergy_optimalheights}
		Let $r_1,\dotsc,r_{M}\in\N$ satisfy $A^{-1}j\leq r_j\leq Aj$ for some positive constant $A$. If the $z_j$ are chosen as in \eqref{eq:optimalheights_proj}, then
		\begin{multline*}
				\expectation{\theta_1,\dotsc,\theta_{M-1}\in[0,2\pi]}{\E_{\log}^{\P}(\Omega_\P(M,\set{r_j}))}
				=\left(N-1\right)^2\log{2}-\frac{r_M}{2}\log r_M-\sum_{j=1}^{M-1}r_j\log{r_j}\\
				-\frac{2N-1}{2}\sum_{j=1}^{M-1}r_j(1-z_j)\log(1-z_j)-\frac{2N-1}{2}\sum_{j=1}^{M-1}r_j(1+z_j)\log(1+z_j)\\
				+\mathrm{Error},\quad |\mathrm{Error}|\leq BM,
		\end{multline*}
where $B$ is a constant depending only on $A$.
	\end{theorem}
\begin{proof}
	Immediate from \cref{eq:energycomparison} and \cref{thm:2.6,thm:energyAntipodal}.
\end{proof}

	\subsection{The projective generalized Diamond ensemble and its energy}\label{sec:diamondproj}
\begin{definition}[Projective generalized Diamond ensemble]\label{def:generalizedDiamondproj}
For any choice of $r(x)$ as in \cref{def:Crx} such that $r(M)$ is an even number, the \emph{projective generalized Diamond ensemble} (\pgde) is the collection of
\begin{equation*}
	N=1+\frac{\alpha_{n+1}+\beta_{n+1} M}{2}+\sum_{\ell=1}^{n}\sum_{j=t_{\ell-1}+1}^{t_\ell}(\alpha_\ell+\beta_\ell j)
\end{equation*}
points given by $\Omega_\P(M,\{r(1),\ldots,r(M)\})$.
As in the case of the sphere, we have $z_j=z(j)$, where $z(x)$ is piecewise defined in each interval $(t_{\ell-1},t_\ell]$ by a quadratic polynomial. We denote this set of points by $\diamondproj(N)$, omitting in the notation the dependence on all the parameters $M,n,t_\ell,\alpha_\ell,\beta_\ell$.
\end{definition}

From \cref{thm:energyAntipodal} and \cref{eq:energycomparison}, the expected value of the logarithmic energy of the \pgde 
follows from that of the associated (spherical) \gde. As in the case of the sphere, this construction is general enough as to describe a collection of $N$ projective points, this time for any $N\geq121$, whose expected energy is as in \cref{thm:maintheorem2proj}. We call this set of points the \emph{quasioptimal projective generalized Diamond ensemble} (\qpgde) and its specific construction is given in \cref{subsec:proof2}.

\subsection{Proof of \cref{thm:maintheorem2proj}}\label{subsec:proof2}
		Take the construction given in the proof of \cref{thm:finalmodel2} for $2N$ points, which clearly imposes that $r_M$ is even. The associated {\pgde} $\diamondproj(N)$  has $N$ points, as claimed. 
The projective logarithmic energy of $\diamondproj(N)$  is, from \cref{thm:finalmodel2}, \cref{thm:energyAntipodal} and \cref{eq:energycomparison}, given by
$$
\frac12\left(W_{\log}(\S^2)(2N)^2-\frac{1}{2}(2N)\log (2N)+2Nc_{\diamond}\right)+ O\left(\sqrt{N}\log N\right)+N^2\log2+\mathrm{Error},
$$
where $|\mathrm{Error}|\leq BM\leq B\sqrt{N}$ for some universal constant $B$. The theorem follows by substituting $W_{\log}(\S^2)=1/2-\log2$ and $W_{\log}(\RP^2)=1-\log2$. \qed

	\section{Auxiliary lemmas}\label{sec:proofs}
	
	\begin{lemma}\label{lemma:energy_parallels}
		Let $x_1,\dotsc,x_{r}\in\S^2$ be $r$ equally spaced points on the parallel of height $z$. Then,
		\begin{equation*}
			-2\sum_{i,k=1}^{r}\log\norm{x_{i}+x_{k}}=
			-2r^2\log(1+z)-2r\log\left(1-\frac{(z-1)^{r}}{(z+1)^{r}}\right).
		\end{equation*}
	\end{lemma}

	\begin{proof}
		By symmetry, we can assume that
		\begin{align*}
			x_{i}&=\left(\sqrt{1-z^2}\cos\frac{2i\pi}{r},\sqrt{1-z^2}\sin\frac{2i\pi}{r},z\right),
		\end{align*}
		so we have
		\begin{align*}
			x_{i}+x_{r}=\left(\sqrt{1-z^2}\left(\cos\frac{2i\pi}{r}+1\right),\sqrt{1-z^2}\sin\frac{2i\pi}{r},2z\right),
		\end{align*}
		and the sum in the lemma equals
		\begin{align*}
			-r\sum_{i=1}^{r}\log\norm{x_{i}+x_{r}}^2&=-r\sum_{i=0}^{r-1}\log\left(2(1-z^2)\bigg(1+\cos\frac{2i\pi}{r}\bigg)+4z^2\right)\\
			&=-r\log\left(2^{r}\prod_{i=0}^{r-1}\left(1+z^2+(1-z^2)\cos\frac{2i\pi}{r}\right)\right).
		\end{align*}
		To compute this product we use \cite[Formula 1.394]{GradshteynRyzhik2007}:
		\begin{equation}\label{Grad:1.394}
			\prod_{i=0}^{n-1} \left\{ x^2 -2xy\cos\frac{2i\pi}{n}+y^2\right\}=x^{2n}-2x^ny^n+y^{2n}=(x^n-y^n)^2,
		\end{equation}
		In our case, we can take
		\begin{equation*}
			x=\frac{z+1}{\sqrt{2}},\quad y=\frac{z-1}{\sqrt{2}},
		\end{equation*}
		which, by \cref{Grad:1.394}, yields
		\begin{equation*}
			\prod_{i=0}^{r-1}\left(1+z^2+(1-z^2)\cos\frac{2k\pi}{r}\right)=\frac{1}{2^{r}}\left((z+1)^{r}-(z-1)^{r}\right)^2.
		\end{equation*}
		Hence, the sum in the lemma equals
		\begin{align*}
			-r\log\left((z+1)^{r}-(z-1)^{r}\right)^2&=-2r\log\left((z+1)^{r}-(z-1)^{r}\right),
		\end{align*}
		as wanted.
	\end{proof}

	\begin{lemma}\label{lemma:lowerboundheights}
		Let $r_1,\dotsc,r_{M}\in\N$ satisfy $A^{-1}j\leq r_j\leq Aj$ for some positive constant $A$, and let the $z_j$ be chosen as in \eqref{eq:optimalheights_proj}. Then, there exists a constant $K>0$ depending only on $A$ such that
		\begin{equation*}
			z_{M-j}\geq\frac{Kj}{M},\quad 1\leq j\leq M-1.
		\end{equation*}
	\end{lemma}

	\begin{proof}
		Note first that $2+A^{-1}M^2\leq N\leq 2+AM^2$. 		We claim that there exists a constant $K_1>0$ such that
		\begin{equation*}
			z_{j}-z_{j+1}\geq \frac{K_1j}{M^2}.
		\end{equation*}
		Indeed, we have
		\begin{align*}
			z_{j}-z_{j+1}&=\frac{r_j+r_{j+1}}{N-1}\geq \frac{A^{-1}j+A^{-1}(j+1)}{N-1}\geq\frac{2A^{-1}j}{(1+A)M^2},
		\end{align*}
		as claimed. Hence,
		\begin{align*}
			z_{M-j}&=z_M+(z_{M-1}-z_{M})+(z_{M-2}-z_{M-1})+\dotsb+(z_{M-j}-z_{M-j+1})\\
				   &\geq\frac{K_1(M-1)}{M^2}+\frac{K_1(M-2)}{M^2}+\dotsb+\frac{K_1(M-j)}{M^2}= \frac{K_1}{M^2}\sum_{k=1}^{j}(M-k)\\
				   &=\frac{K_1jM}{M^2}-\frac{K_1j(j+1)}{2M^2}=\frac{K_1j}{M^2}\left(M-\frac{j+1}{2}\right)\geq \frac{K_1j}{2M},
		\end{align*}
		and the lemma follows.
	\end{proof}
	
	\begin{lemma}\label{lemma:1}
		Let $r_1,\dotsc,r_{M}\in\N$ satisfy $A^{-1}j\leq r_j\leq Aj$ for some positive constant $A$, and let the $z_j$ be chosen as in \eqref{eq:optimalheights_proj}. Then, there exists a constant $\Upsilon>0$ depending only on $A$ such that
		\begin{equation*}
			\abs*{\sum_{j=1}^{M-1}r_j\log\bigg(1-\frac{(z_j-1)^{r_j}}{(z_j+1)^{r_j}}\bigg)}\leq \Upsilon M.
		\end{equation*}
	\end{lemma}
	
	\begin{proof}
		Let $S$ denote the quantity we want to bound. We have
		\begin{align*}
			S&\leq \abs*{\sum_{j=1,\ \text{$r_j$ odd}}^{M-1} r_j\log\left(1+\left(\frac{1-z_j}{1+z_j}\right)^{r_j}\right)}+\abs*{\sum_{j=1,\ \text{$r_j$ even}}^{M-1} r_j\log\left(1-\left(\frac{1-z_j}{1+z_j}\right)^{r_j}\right)}\\
			&\leq \abs*{\sum_{j=1,\ \text{$r_j$ odd}}^{M-1} r_j\log\left(1-\left(\frac{1-z_j}{1+z_j}\right)^{r_j}\right)}+\abs*{\sum_{j=1,\ \text{$r_j$ even}}^{M-1} r_j\log\left(1-\left(\frac{1-z_j}{1+z_j}\right)^{r_j}\right)}\\
			&=-\sum_{j=1}^{M-1} r_j\log\left(1-\left(\frac{1-z_j}{1+z_j}\right)^{r_j}\right)=-\sum_{j=1}^{M-1} r_{M-j}\log\left(1-\left(\frac{1-z_{M-j}}{1+z_{M-j}}\right)^{r_{M-j}}\right),
		\end{align*}
		where we have used that $\log(1+x)\leq-\log(1-x)$ for $0<x<1$. From \cref{lemma:lowerboundheights}, we have
		\begin{align*}
			S&\leq -A\sum_{j=1}^{M-1} (M-j)\log\left(1-\left(\frac{1-Kj/M}{1+Kj/M}\right)^{A^{-1}(M-j)}\right)\\
			&=-A\sum_{j=1}^{M-1} (M-j)\log\left(1-\left(\frac{M-Kj}{M+Kj}\right)^{A^{-1}(M-j)}\right).
		\end{align*}
		Recall the following inequality:
		\begin{equation}\label{eq:ineqlog}
			-\log(1-x)\leq Cx,\quad 0<x<a<1,\quad C>0,\ C=C(a).
		\end{equation}
		Note that
		\begin{align*}
			\left(\frac{M-Kj}{M+Kj}\right)^{A^{-1}(M-j)}&=\left(1+\frac{-2Kj}{M+Kj}\right)^{(M+Kj)\frac{A^{-1}(M-j)}{M+Kj}}\leq e^{-2Kj\frac{A^{-1}(M-j)}{M+Kj}}\\
			&\leq e^{-\widetilde{K}j(M-j)/M}\leq e^{-\widetilde{K}(M-1)/M}\leq e^{-\widetilde{K}/2}<1,
		\end{align*}
		where $\widetilde{K}=2KA^{-1}/(1+K)$. Then, using \eqref{eq:ineqlog} we have
		\begin{align*}
			S&\leq AC\sum_{j=1}^{M-1} (M-j)\left(\frac{M-Kj}{M+Kj}\right)^{A^{-1}(M-j)}\leq AC\sum_{j=1}^{M-1} (M-j)e^{-\widetilde{K}j(M-j)/M}.
		\end{align*}
		Let $\alpha=e^{-\widetilde{K}}<1$. Without loss of generality, we can assume $M>2$. Then,
		\begin{align*}
			S&\leq AC\sum_{j=1}^{M-1} (M-j)\alpha^{j(M-j)/M}\leq 2ACM\sum_{j=1}^{\floor{M/2}+1} \alpha^{j(M-j)/M}\\
			&\leq 2ACM \sum_{j=1}^{\floor{M/2}+1}\alpha^{jD}\leq  2ACM \sum_{j=1}^{\infty}\alpha^{jD}= 2ACM\frac{\alpha^{D}}{1-\alpha^{D}},
		\end{align*}
		where $D>0$ is a constant. The lemma follows.
	\end{proof}	


\bibliographystyle{amsplain}

\end{document}